\documentclass[12pt,reqno]{amsart}
\usepackage{amsmath}
\usepackage[english,  activeacute]{babel}
\usepackage[latin1]{inputenc}
\usepackage{amssymb}
\usepackage{amsthm}
\usepackage{graphics,graphicx}
\usepackage{array}
\usepackage{a4wide}
\allowdisplaybreaks
\usepackage{color, url}
\usepackage{float}
\usepackage{multicol}
\usepackage[shortlabels]{enumitem}
\usepackage[hypertexnames=false,colorlinks=true,allcolors=blue]{hyperref}
\theoremstyle{plain}
\newtheorem{theorem}{Theorem}[section]

\newtheorem{corollary}[theorem]{Corollary}
\newtheorem{lemma}[theorem]{Lemma}
\theoremstyle{definition}
\newtheorem{definition}[theorem]{Definition}

\begin{document}
		\title[A Combinatorial Proof of an Identity of Malik and Sarma on Overpartitions with Repeated Smallest Non-Overlined Part ]{A Combinatorial Proof of an Identity of Malik and Sarma on Overpartitions with Repeated Smallest Non-Overlined Part } 
	
		\author{Suparno Ghoshal and  Arijit Jana }
\address{School of Computer Science, University of Petroleum and Energy Studies (UPES), Dehradun, Uttarakhand, India}
	\email{ghoshalsuparno1331@gmail.com }
	\address{Department of Mathematics, National Institute of Technology Silchar, Assam 788010, India}
	\email{jana94arijit@gmail.com}

	
	\begin{abstract}
		Motivated by the work of Andrews and El Bachraoui on partitions with repeated smallest parts, Malik and Sarma recently extended this concept to overpartitions. Using generating functions and $q$-series techniques, they established several results, including four identities, and posed the problem of finding combinatorial proofs for these identities. Recently, Baruah, Li, and Mohanta provided combinatorial proofs for three of these identities. In this paper, we present a combinatorial proof of the remaining identity.
		\end{abstract}
	\maketitle
\section{Introduction and statement of results}
Let $a$ be any complex number and $n$ be nonnegative integer. Throughout this paper, we adopt the following standard notation:
	\begin{align*}
	(a;q)_n &:= \prod_{j=0}^{n-1} (1 - a q^{j}), \\
	(a;q)_\infty &:= \lim_{n \to \infty} (a;q)_n, \qquad |q| < 1.
	\end{align*}
	Recently, Andrews and El Bachraoui \cite{AndrewsElBachraoui2025} investigated integer partitions in which the smallest part occurs exactly $k$ times and all remaining parts are distinct. Let $\mathrm{sptk}_{d}(n)$ denote the number of such partitions of $n$. They showed that the associated generating functions can be written as linear combinations of $q$-Pochhammer symbols with polynomial coefficients in $q$. An overpartition is a partition in which the first occurrence of a part is allowed to be overlined; for further details, see \cite{CorteelLovejoy2004}. Motivated by the work in \cite{AndrewsElBachraoui2025}, Malik and Sarma \cite{Malik26} recently extended this concept to overpartitions. They introduced the following functions.

\begin{definition}\cite[Page 2]{Malik26}\label{def1}
	Let $\overline{\mathrm{Spt}}k(n)$ denote the set of overpartitions of $n$ where the smallest non-overlined part, say $s(\pi)$, appears $k$ times and every overlined part is bigger than $s(\pi)$. Accordingly, let $\overline{\mathrm{spt}}k(n)$ be the cardinality of $\overline{\mathrm{Spt}}k(n)$.
\end{definition}
In \cite{Malik26}, Malik and Sarma gave the following result, which is a direct analogue of Theorem 1 in  \cite{AndrewsElBachraoui2025}

\begin{theorem} \cite[Theorem 2.1]{Malik26}\label{MStheo1}
	For any positive integer $k$, let $\overline{\mathrm{spt}}k(n)$ denote the partition function as in Definition \ref{def1}. Then,
	\begin{align*}
		\sum\limits_{n=1}^{\infty}\overline{\mathrm{spt}}k(n)q^n=\overline{P}_k(q)\frac{(-q^2;q)_{\infty}}{(q^2;q)_{\infty}}+(-1)^k\frac{(q;q)_{k-1}}{(-q;q)_{k}},
	\end{align*}
	where 
	\begin{align*}
		\overline{P}_k(q)=\begin{cases}
			1 & \text{if } k=1, \\
			\frac{(q^{k-1}-1)\overline{P}_{k-1}(q)+q^{k-1}(1+q)}{1+q^k}& \text{if } k>1.
		\end{cases}
	\end{align*}
\end{theorem}
The case $k= 2$ yields a result involving multiple combinations of the spt-functions and the corresponding subclasses of overpartitions. We state below the identity for  $\overline{\mathrm{spt}}2(n).$ 

\begin{corollary} \cite[Corollary 4.2]{Malik26}
\label{cor:corsptk2}
For $n>1$, we have $$\overline{\mathrm{spt}}2(n)+\overline{\mathrm{spt}}2(n-1)+\overline{\mathrm{spt}}2(n-2)+\overline{\mathrm{spt}}2(n-3)+\overline{p}_{uu}(n)=2\overline{p}_{u}(n-1),$$
where $\overline{p}_{uu}(n)$ denotes the number of overpartitions of $n$ with no non-overlined $1$'s and $2$'s and $\overline{p}_{u}(n)$ denotes the number of overpartitions of $n$ with no non-overlined $1$'s.
\end{corollary}
They also deduced three other similar corollaries in \cite[Section 4]{Malik26}. In the concluding remarks of their paper, they posed the problem of finding combinatorial proofs of these corollaries. Subsequently, Baruah, Li, and Mahanta \cite{Baruah26} provided combinatorial proofs for three of them. However, Corollary~\ref{cor:corsptk2} remains open. In this paper, we provide a combinatorial proof of the  Corollary~\ref{cor:corsptk2}.
\section{Proof of the Main result}
We begin by defining the set $A(n)$ that consists of all possible overpartitions of $n$, having no part with part value $1$ and no non-overlined part of value $2$. Now, we provide a straightforward lemma,

\begin{lemma}\label{lem1}
    $|A(n)| + |A(n - 1)| = \overline{p}_{uu}(n)$.
\end{lemma}
\begin{proof}
   If $\overline{P}_{uu}(n)$ is the set of overpartitions of $n$ with no non-overlined $1$'s and $2$'s, then define $\overline{P}^{0}_{uu}(n)$ (and $\overline{P}^{1}_{uu}(n)$) to be its subsets having no ones (and with overlined one) respectively. Since the sets $\overline{P}^{0}_{uu}(n)$ and $\overline{P}^{1}_{uu}(n)$ are mutually exclusive, we have $\overline{p}_{uu}(n) = |\overline{P}^{0}_{uu}(n) \cup \overline{P}^{1}_{uu}(n)|= |\overline{P}^{0}_{uu}(n)| + |\overline{P}^{1}_{uu}(n)|$.\\
   Now we try to give a bijection $\phi : A(n) \cup A(n - 1) \to \overline{P}^{0}_{uu}(n) \cup \overline{P}^{1}_{uu}(n)$. If $\pi \in A(n)$, then define $\phi(\pi) = \lambda$, where $\lambda \in \overline{P}^{0}_{uu}(n)$ is the overpartition obtained from $\pi$ by the following rule.
   \begin{enumerate}
       \item Use the $\phi$ mapping as an identity mapping. 
   \end{enumerate}
 In addition, if $\kappa \in A(n - 1)$, then let $\phi(\kappa) = \omega$, where $\omega \in \overline{P}^{1}_{uu}(n)$ is obtained by the following rule. 
 \begin{enumerate}
     \item Add an overline part of the value $1$ to $\kappa$.
 \end{enumerate}
    Further, we need to prove that $\phi$ has an inverse in order to establish that it is a bijection. Since $\phi$ acts as an identity mapping from $A(n) \to \overline{P}^{0}_{uu}(n)$, it is always bijective. In the other scenario when $\phi$ maps from the set $A(n - 1)$ to $\overline{P}^{1}_{uu}(n)$, we can define $\phi^{- 1}(\omega) = \kappa$ where $\omega \in \overline{P}^{1}_{uu}(n)$ and $\kappa \in A(n - 1)$ as follows.
    \begin{enumerate}
        \item Remove the overlined part valued $1$ from $\omega$.
    \end{enumerate}
    This completes the proof of the above stated lemma.
\end{proof}
\begin{lemma}\cite[Corollary 4.1]{Malik26}\label{lem2}
    $\overline{spt}1(n) + \overline{spt}1(n - 1) = \overline{p}_{u}(n)$.
\end{lemma}
A combinatorial proof of this lemma was provided by Baruah, Li, and Mahanta in \cite{Baruah26}. \\
In this paper, we provide the following theorem, which will serve as a stepping stone in our journey towards proving Corollary \ref{cor:corsptk2}.
\begin{theorem}\label{thm1}
    $2\overline{spt}1(n) = \overline{spt}2(n - 1) + \overline{spt}2(n + 1) + |A(n + 1)|$.
\end{theorem}
\begin{proof}
    Our proof of the above theorem is inspired by Keith and Sagan's \cite{KeithSagan} combinatorial proof of a theorem by Banerjee, Bringmann, and Dixit \cite{BBD26}. Take the set $\overline{Spt}1(n)$ and make a copy of the set. Now, color all overpartitions of the original set blue, while coloring the overpartitions of the copied set green. Denote those two sets as $\overline{Spt}^{B}1(n)$ (and $\overline{Spt}^{G}1(n)$) respectively. This helps reduce the problem in Theorem \ref{thm1} to 
    \begin{align}
     \overline{spt}^{B}1(n) + \overline{spt}^{G}1(n) = \overline{spt}2(n - 1) + \overline{spt}2(n + 1) + |A(n + 1)|.
     \end{align}
     It remains only to establish a bijection
     \[
     \zeta: \overline{Spt}^{B}1(n) \cup \overline{Spt}^{G}1(n) \to \overline{Spt}2(n - 1) \cup \overline{Spt}2(n + 1) \cup A(n + 1).
\]
     For an overpartition $\alpha \in \overline{Spt}^{B}1(n)$, we define the mapping for two different kinds of partitions depending on whether the smallest part of the partition is equal to $1$ or strictly greater than $1$. \\
     {\bf Case 1 ($s(\alpha) = 1$):} In this case define $\zeta(\alpha) = \mu$, where $\mu \in \overline{Spt}2(n + 1)$ is achieved by the following method.
     \begin{enumerate}
         \item Simply add another non-overlined part valued $1$ to $\alpha$, i.e. $\alpha \mapsto \alpha, 1$ and remove the blue color.
     \end{enumerate}
     {\bf Case 2 ($s(\alpha) \geq 2$):} 
     Define $\zeta(\alpha) = \mu$, where $\mu \in A(n + 1)$ is obtained by the following rule.
     \begin{enumerate}
         \item Take the non-overlined $s(\alpha)$ and map it to $s(\alpha) + 1$ and keep it non-overlined and remove the color. 
     \end{enumerate}
     Now, we take the partitions $\beta \in \overline{Spt}^{G}1(n)$ and define $\zeta(\beta) = \gamma$, where $\gamma \in \overline{Spt}2(n - 1) \cup \overline{Spt}2(n + 1) \cup A(n + 1)$ is obtained using the following method.\\
     {\bf Case 1 ($s(\beta) + 1 \text{ is not a part of } \beta$) :} Take the smallest part $s(\beta)$, and convert it into an overlined part having part value $s(\beta) + 1$. The transformed overpartition belongs to the set $A(n + 1)$. \\
     {\bf Case 2 (non-overlined part $s(\beta) + 1$ is part of $\beta$):} Just reduce the part value of the part $s(\beta) + 1$ by $1$. The overpartition after the mapping definitely belongs to the set $\overline{Spt}2(n - 1)$.\\
     {\bf Case 3 (exactly one part $s(\beta) + 1$ is part of $\beta$ and it is overlined):} In this case, remove the overline from the $s(\beta) + 1$ part and add one to the smallest part i.e. $(s(\beta) \mapsto s(\beta) + 1)$. The new overpartition will be part of the set $\overline{Spt}2(n + 1)$. \\
    In addition, we provide the inverse of $\zeta$ which will guarantee that $\zeta$ is a bijection. \\
    {\bf Case 1(i) (Assume $\gamma \in A(n + 1)$ and $s(\gamma)$ occurs at least once non-overlined) :} Take the $s(\gamma)$, subtract $1$ from it and color it with blue.\\
    {\bf Case 1(ii) (Assume $\gamma \in A(n + 1)$ and $s(\gamma)$ occurs only overlined)} Subtract $1$ from $s(\gamma)$, remove the overline, and color it green.\\
    {\bf Case 2 (Let $\gamma \in \overline{Spt}2(n - 1)$)} Simply do the following $s(\gamma) \mapsto s(\gamma) + 1$ and color it green.\\
    {\bf Case 3(i) (Let $\gamma \in \overline{Spt}2(n + 1)$ and $s(\gamma) = 1$ and $s(\gamma)$ is not overline) :} Remove the $1$ from $\gamma$ and color the rest of the partition with blue.\\
    {\bf Case 3(ii) (Let $\gamma \in \overline{Spt}2(n + 1)$ and $s(\gamma) > 1$) :} Replace one $s(\gamma)$ with $s(\gamma) - 1$.
    while adding an overline on the other $s(\gamma) \mapsto \overline{s(\gamma)}$. 
    Color every part with green. 

    The above mentioned mapping will serve as $\zeta^{-1}$. 
     Since, $\zeta: \overline{Spt}^{B}1(n) \cup \overline{Spt}^{G}1(n) \to \overline{Spt}2(n - 1) \cup \overline{Spt}2(n + 1) \cup A(n + 1)$ is a bijective mapping, we have our desired identity. 
\end{proof}
\begin{proof}[Proof of Corollary \ref{cor:corsptk2}:]
    In Theorem \ref{thm1}, first replace $n$ with $n - 1$, then replace $n$ with $n - 2$, which gives us the following,
    $$2\overline{spt}1(n - 1) = \overline{spt}2(n - 2) + \overline{spt}2(n) + |A(n)|.$$ \\
    $$2\overline{spt}1(n - 2) = \overline{spt}2(n - 3) + \overline{spt}2(n - 1) + |A(n - 1)|.$$
    Adding the above equations we get,
    $$2(\overline{spt}1(n - 1) + \overline{spt}1(n - 2)) = \overline{spt}2(n - 2) + \overline{spt}2(n) + |A(n)| + \overline{spt}2(n - 3) + \overline{spt}2(n - 1) + |A(n - 1)|.$$
    Now, using Lemma \ref{lem1} and \ref{lem2}, we reduce the above equation as below,
    $$2 \overline{p}_{u}(n - 1) = \overline{spt}2(n - 2) + \overline{spt}2(n) + \overline{spt}2(n - 3) + \overline{spt}2(n - 1) + \overline{p}_{uu}(n).$$
    This completes the combinatorial proof of Corollary \ref{cor:corsptk2}.  
\end{proof}

{\bf Data Availability Statement}\\
Data sharing is not applicable to this article as no datasets were generated or analyzed during the current study.\\
{\bf Declarations}\\
		{\bf Author contributions:} All authors have contributed equally to the preparation of this manuscript.\\
{\bf 	Conflict of interest}
The authors declare that there is no conflict of interest.\\
{\bf  Funding} : The authors did not receive any funding for this research.

\end{document}